\documentclass[12pt,reqno]{amsart}
\usepackage{pgfplots}
\usepackage{tikz}
\pgfplotsset{compat=newest}
\pgfplotsset{ytick style={draw=none}}
\pgfplotsset{xtick style={draw=none}}
\usepackage{amsmath,amsfonts,amssymb,amsthm}
\usepackage{graphicx}
\graphicspath{ }
\usepackage{mathrsfs}
\usepackage{epstopdf}
\usepackage{csquotes}
\usepackage{wrapfig}
\usepackage{accents}
\usepackage{caption}
\usepackage{subcaption}
\usepackage{calligra}
\usepackage[colorlinks]{hyperref}
\usepackage{kantlipsum}
\usepackage{cleveref}
\hypersetup{colorlinks, 
	breaklinks,
	linkcolor=red,
	citecolor=red,
	linktocpage=true}

\numberwithin{figure}{section}

\theoremstyle{plain}
\newtheorem{thm}{Theorem}[section]
\newtheorem{lem}[thm]{Lemma}

\theoremstyle{definition}
\newtheorem{defn}{Definition}[section]

\theoremstyle{remark}
\newtheorem{rem}{Remark}

\usepackage{mathtools}
\makeatletter
\@namedef{subjclassname@2020}{%
	\textup{2020} Mathematics Subject Classification}
\makeatother
\title[BK-module via $p$-divisible group]{%
	Properties of Breuil-Kisin modules inherited by $p$-divisible groups}

\author[M. A. Sarkar]{MABUD ALI SARKAR}
\address{Department of Mathematics\\ The University of Burdwan \\ Burdwan-713101, India.}
\email{mabudji@gmail.com}
\author[A. A. Shaikh]{Absos Ali Shaikh} 
\address{Department of Mathematics,\\ The University of Burdwan,\\  Burdwan-713101, India.}
\email{aashaikh@math.buruniv.ac.in}

\begin{document} 
	
	\begin{abstract}
		In this paper, by assuming a faithful action of a finite flat $\mathbb{Z}_p$-algebra $\mathscr{R}$ on a $p$-divisible group $\mathcal{G}$ defined over the ring of $p$-adic integers $\mathscr{O}_K$, we construct a category of new Breuil-Kisin module $\mathfrak{M}$ defined over the ring $\mathfrak{S}:=W(\kappa)[\![u]\!]$ and study the freeness and projectiveness properties of such a module. 
	\end{abstract}
	
	\subjclass[2020]{11F80, 11F85, 14L05}
	\keywords{p-adic numbers, Breuil-Kisin module, p-divisible group, p-adic Tate module}
	\maketitle
	\section{Introduction and motivation} \label{s1}
	For a fixed prime numebr $p$, consider the $p$-adic field $\mathbb{Q}_p$ with ring of integers $\mathbb{Z}_p$ and residue field $\mathbb{F}_p$. Let $\pi \in \mathscr{O}_K$ be an uniformizer of the ring of integers $\mathscr{O}_K$ in a finite extension $K$ of $\mathbb{Q}_p$ with $\kappa \cong \mathscr{O}_K/\pi \mathscr{O}_K$ being the residue field of $K$. Let $E(u)$ be the minimal (Eisenstein) polynomial of $\pi$ over $W(\kappa)$, where $W(\kappa)$ is the ring of Witt vectors of the residue field $\kappa$. We denote by $\mathfrak{S}:=W(\kappa)[\![u]\!]$ as the ring of power series over $W(\kappa)$. Let $\varphi_{\mathfrak{S}}: W(\kappa)[\![u]\!] \to W(\kappa)[\![u]\!]$ be the ring endomorphism which extends the Frobenious on $W(\kappa)$ satisfying $\varphi_{\mathfrak{S}}(u)=u^p$.

	Classification of $p$-divisible groups or Barsotti-Tate groups is a useful tool in the study of arithmetic geometry and $p$-adic Hodge theory. For example, given an abelian scheme $S$, consider the inductive system $\{(S[p^n]), S[p^n] \to S[p^{n+1}] )\}_{n}$, where $S[p^n]$ denotes the $p^n$-power torsion points on $S$. This gives rise to the $p$-divisible group $S[p^{\infty}]$, which inherits lots of information about the Abelian scheme. That is why it is important to study $p$-divisible groups. There are several classifications of $p$-divisible groups through linear algebraic data. One such category of important linear algebraic objects that classifies the $p$-divisible groups is the category of Breuil-Kisin modules defined over the ring of Witt vectors $\mathfrak{S}=W(\kappa)[\![u]\!]$. In fact, the category of $p$-divisible groups $\text{BT}(\mathscr{O}_K)$ over $\mathscr{O}_K$ is equivalent to the category of Breuil-Kisin modules $\text{BT}_{/\mathfrak{S}}^{\varphi}$, which was conjectured by Breuil \cite{Bre} and proved by Kisin \cite{Kis}. More precisely, the Breuil-Kisin module is an important tool in studying Galois representations.

	Let us consider the Breuil-Kisin module $\mathfrak{M}$ over $\mathfrak{S}$.  For a finite flat $\mathbb{Z}_p$-algebra $\mathscr{R}$, consider the $p$-divisible group $\mathcal{G}$ defined over $\mathscr{O}_K$ endowed with a faithful $\mathscr{R}$-action. Since the category of such modules $\mathfrak{M}$ is equivalent to the category of $p$-dvisible groups $\mathcal{G}$ over $\mathscr{O}_K$, the above $\mathcal{G}$ equipped with $\mathscr{R}$-action results a \textit{newly constructed}  $\mathfrak{S}_\mathscr{R}:=(\mathscr{R} \otimes_{\mathbb{Z}_p} W(\kappa)) [\![u]\!]$-module $\mathfrak{M}$. Therefore, we study the freeness and projectiveness properties of this newly produced Breuil-Kisin module over $(\mathscr{R} \otimes_{\mathbb{Z}_p}W(\kappa))[\![u]\!]$. However, the projectiveness and freeness property of the $(\mathscr{R} \otimes_{\mathbb{Z}_p} W(\kappa)) [\![u]\!]$-module $\mathfrak{M}$ does not follow for general $\mathscr{R}$. For example, consider the ring $\mathscr{R}=\mathbb{Z}_p[u]/(u^2)$. This acts as $0$ on any Breuil-Kisin module, and so the Breuil-Kisin module would not be free or even projective over $(\mathscr{R} \otimes_{\mathbb{Z}_p}W(\kappa))[\![u]\!]$ for such $\mathscr{R}=\mathbb{Z}_p[u]/(u^2)$. This suggests us $\mathscr{R}$ is a regular local ring. Another obstruction is that spectra $Spec(\mathscr{R})$ of $\mathscr{R}$ might not be connected. In that case, $\mathfrak{M}$ could be projective, but have different ranks on different components of spectra, and hence $\mathfrak{M}$ can not be free. Further the ring $(\mathscr{R} \otimes_{\mathbb{Z}_p} W(\kappa)) [\![u]\!]$ is not necessarily local, for example $W(\kappa) \otimes_{\mathbb{Z}_p} W(\kappa)$ decomposes as a direct product of $[\kappa:\mathbb{F}_p]$ copies of $W(\kappa)$. 
	Therefore, the natural question arises: \\ \hspace*{1 cm}$\text{ When is the Breuil-Kisin module $\mathfrak{M}$ projective and free over $(\mathscr{R} \otimes_{\mathbb{Z}_p} W(\kappa)) [\![u]\!]$ ? }$ \\ Kisin answers the above question for a few special cases (see Theorems \ref{t2.1}, \ref{t2.2}, \ref{T2.3}). There are more such considerations in (\cite{BL}, \cite{BL1}).
	The present paper provides some more details in the Subsection \ref{s2.1}.
	\section{Projectiveness and freeness of the objects in $\text{Mod}_{/\mathfrak{S}_\mathscr{R}}^{\varphi}$} \label{s4}
	Before we start our investigations on the projectiveness and freeness properties of the modules $\mathfrak{M} \in \text{Mod}_{/\mathfrak{S}_\mathscr{R}}^{\varphi}$, we recall some definitions:
	\begin{defn}
		A Breuil-Kisin module (BK-module for brevity) of height $h$ (positive integer) is a finitely generated and free $\mathfrak{S}$-module $\mathfrak{M}$ endowed with a $\varphi_{\mathfrak{S}}$-semilinear endomorphism $\varphi_{\mathfrak{M}}: \mathfrak{M} \to \mathfrak{M}$ such that $E(u)^h$ kills the cokernel $\mathfrak{M}/(1 \otimes \varphi_{\mathfrak{M}})(\mathfrak{M})$ (i.e., $\varphi_{\mathfrak{M}}$ is injective) of the  $\mathfrak{S}$-linearization $1 \otimes \varphi_{\mathfrak{M}}: \varphi_{\mathfrak{S}}^{*}(\mathfrak{M}) \to \mathfrak{M}$ of $\varphi_{\mathfrak{M}}$, where we denote $\varphi_{\mathfrak{S}}^{*}(\mathfrak{M}):=\mathfrak{S} \otimes_{\varphi_{\mathfrak{S}}, \mathfrak{S}}  \mathfrak{M}$. We also denote the category of such modules by $\text{Mod}_{/\mathfrak{S}}^{\varphi}$. For $h=1$, we call such modules as height $1$ Breuil-Kisin modules and such a category is denoted by $\text{BT}_{/\mathfrak{S}}^{\varphi}$ as mentioned above. In other words, $\text{BT}_{/\mathfrak{S}}^{\varphi}$ is a full subcategory of $\text{Mod}_{/\mathfrak{S}}^{\varphi}$.  \\
	\end{defn}
	\begin{rem}
		In literature, there are also slightly different versions of Breuil-Kisin modules used by different authors (e.g., \cite{Cai}, \cite{Gao}) in different perspectives. Some authors preferred only either the freeness property or the finiteness property. Some authors assumed $\mathfrak{M}$ to be $u$-torsion free while some authors inverted $E(u)$ in defining $\mathfrak{S}[1/E(u)]$-linearization map $\mathfrak{S} \otimes_{\varphi, \mathfrak{S}} \mathfrak{M}[1/E(u)] \to \mathfrak{M}[1/E(u)]$. Simply, we can think of the Breuil-Kisin module as a module over the ring $\mathfrak{S}$ with some impositions. The best way to understand the Breuil-Kisin modules (linear algebra data) is to look at the corresponding objects in the Galois representations such as crystalline Galois representations with different Hodge-Tate weights or torsion Galois representations. 
	\end{rem}
	\begin{defn}
		Let us denote $\text{Mod}_{/\mathfrak{S}_\mathscr{R}}^{\varphi}$ as the category of those modules from the category $\text{Mod}_{/\mathfrak{S}}^{\varphi}$ equipped with $\mathscr{R}$-action. That is every member, say $\mathfrak{M} \in \text{Mod}_{/\mathfrak{S}_\mathscr{R}}^{\varphi} $ is a $\mathfrak{S}_\mathscr{R}:=(\mathscr{R} \otimes_{\mathbb{Z}_p} W(\kappa)) [\![u]\!]$-module.  Then $\text{Mod}_{/\mathfrak{S}_\mathscr{R}}^{\varphi}$ forms a subcategory of $\text{Mod}_{/\mathfrak{S}}^{\varphi}$. 
	\end{defn}
	\begin{defn} \cite{Tate}
		Let $p$ be a prime number, and $d$ be a non-negative integer. A $p$-divisible group $\mathcal{G}$ over the ring of integers $\mathscr{O}_K$ of height $d$ is the inductive system $$\mathcal{G}=(\mathcal{G}_m,i_m), \ m \geq 0$$ with group scheme homomorphism $i_m: \mathcal{G}_m \to \mathcal{G}_{m+1}$, and satisfying the following criterions:
		\begin{enumerate}
			\item[(i)] $\mathcal{G}_m$ is a finite flat group scheme over the ring of integers $\mathscr{O}_K$ of order $p^{md}$ for all $m$,
			\item[(ii)] the sequence $$ 0 \to \mathcal{G}_m \xrightarrow{i_m} \mathcal{G}_{m+1} \xrightarrow{p^m}\mathcal{G}_{m+1}$$ is exact.
			We usually denote by $\text{BT}(\mathscr{O}_K)$ the category of $p$-divisible groups over $\mathscr{O}_K$. 
		\end{enumerate}
	\end{defn}

	\begin{defn} \cite{Tate}
		Let $\mathcal{G}$ be a $p$-divisible group over the ring of integers $\mathscr{O}_K$. Let $$ \cdots \xleftarrow{i_3}\mathcal{G}_3 \xleftarrow{i_2} \mathcal{G}_2 \xleftarrow{i_1}\mathcal{G}_1$$ be the inverse system associated to the $p$-divisible group $\mathcal{G}$. Then the $p$-adic Tate module of $\mathcal{G}$ is $$ T_p(\mathcal{G}):=\varprojlim_{m} \mathcal{G}_m(K^{alg}),$$ where we take the inverse limit over positive integers $m$ with the transition maps $i_m$ and $K^{alg}$ is the separable closure of $K$. Note that $T_p(\mathcal{G})$ is free $\mathbb{Z}_p$-module and it encodes all the $p$-power torsion points of $\mathcal{G}$ in $K^{alg}$. 
	\end{defn}

	Before producing the main results, to help the readers, we have collected the following three results (Theorem~ \ref{t2.1}, Theorem~\ref{t2.2} and Theorem \ref{T2.3}), which already appears in the literature. 
	\begin{thm} \label{t2.1}
		If the finite flat $\mathbb{Z}_p$-algebra $\mathscr{R}$ contains all the embeddings $\sigma$ of $W(\kappa)$, then $\mathfrak{M}$ is free $(\mathscr{R} \otimes_{\mathbb{Z}_p}W(\kappa))[\![u]\!]$-module.
	\end{thm}
	
	\begin{proof}
		The torsion version of the above result appears in \cite[Lemma.~1.2.2] {Kis2}.
	\end{proof}
	
	\begin{thm} \label{t2.2}
		If $\mathscr{R}$ be the regular local ring $\mathscr{O}_K$ in finite extension $K$ of $\mathbb{Q}_p$, then $\mathfrak{M}$ is $(\mathscr{R} \otimes_{\mathbb{Z}_p} W(\kappa)) [\![u]\!]$-projective module. 
	\end{thm}
	
	\begin{proof}
		The torsion version of the above result appears in \cite[Lemma.~1.2.2] {Kis2}.
		It also appears in a different perspective and a disguised form (in the middle of the proof of \cite[Proposition~1.6.4]{Kis1}).
	\end{proof}

	We can recover the $p$-adic Tate module from Breuil-Kisin module easily by taking the $\varphi$-fixed points of $\mathfrak{M} \otimes_{\mathfrak{S}} W(\text{Frac}(\mathfrak{R}))$ i.e., $$T_p(\mathcal{G})=\left(\mathfrak{M} \otimes_{\mathfrak{S}} W(\text{Frac}(\mathfrak{R}))\right)^{\varphi=1} ,$$where $\text{Frac}(\mathfrak{R})$ defined below. From the work of \cite{Kis}, we also have an inverse process. That is, we start with a $p$-adic Tate module $T_p(\mathcal{G})$ of the $p$-divisible group $\mathcal{G}$ and get a Breuil-Kisin module. Hence the further question remains to check:
	$$\text{Is $(\mathscr{R} \otimes_{\mathbb{Z}_p} W(\kappa))$-module $\mathfrak{M}$ free if the $p$-adic Tate module $T_p(\mathcal{G})$ is free as $\mathscr{R}$-module ?}$$ 
	We can answer this question by lifting the category $\text{Mod}_{/\mathfrak{S}_\mathscr{R}}^{\varphi}$ of Breuil-Kisin module to \'etale level i.e., we consider the associated $\acute{e}$tale $\varphi$-modules.
		\begin{thm} \label{T2.3}
			If the $p$-adic Tate module $T_p(\mathcal{G})$ of a $p$-divisible group $\mathcal{G}$ over $\mathscr{O}_K$ is free then the associated $\acute{e}$tale $\varphi$-module is free.
		\end{thm}
	\begin{proof}
		The proof is a consequence of \cite[Lemma~1.2.7]{Kis2}.
	\end{proof}
	\subsection{Main Results} \label{s2.1}
	Throughout the paper, the term finite module means finitely generated module.
	As a consequence of Theorem \ref{t2.2}, we prove the following result:
	\begin{thm}
		Let $\mathscr{R}=\mathscr{O}_K$ be the ring of integers in the finite extension $K$ of $\mathbb{Q}_p$ and assume, there is a finite flat map $\mathscr{R} \to \mathscr{R}'$, where $\mathscr{R}'$ is also a regular local ring. Then any $(\mathscr{R}' \otimes_{\mathbb{Z}_p}W(\kappa))[\![u]\!]$-module $\mathfrak{M}$, which is finite free ((i.e., finitely generated and free) over $(\mathscr{R} \otimes_{\mathbb{Z}_p}W(\kappa))[\![u]\!]$, is also finite free $(\mathscr{R}' \otimes_{\mathbb{Z}_p}W(\kappa))[\![u]\!]$-module.
	\end{thm}
	\begin{proof}
		It is sufficient to show that if $\mathscr{R} \to \mathscr{R}'$ is a finite flat morphism between regular local rings and if $\mathfrak{M}$ is an $\mathscr{R}'$-module that is finite free over $\mathscr{R}$, then $\mathfrak{M}$ is a finite free module over $\mathscr{R}'$ because, 
		let $\mathscr{R}=\mathscr{O}_K$ be the ring of integers of the finite extension $K$ of $ \mathbb{Q}_p$. Then $\mathscr{R} \otimes_{\mathbb{Z}_p} W(\kappa)$ is a direct product of rings of integers of finite extensions of $\mathbb{Q}_p$.  
		Hence it is sufficient to show that if $L$ is a finite extension of $\mathbb{Q}_p$ with ring of integers $\mathscr{O}_L$ then the $\mathscr{O}_L[\![u]\!]$-module $\mathfrak{M}$, a finite free $\mathbb{Z}_p[\![u]\!]$-module, is also a free as $\mathscr{O}_L[\![u]\!]$-module. When $\mathfrak{M}=0$, the statement is trivial. Let $\mathfrak{M}$ be non-zero. Since $\mathscr{O}_L[\![u]\!]$ is a regular local ring, all finitely generated modules have finite projective dimensions. So by the Auslander-Buchsbaum formula, we have
		\begin{equation} \label{e1}
		\text{projdim}(\mathfrak{M})+\text{depth}(\mathfrak{M})=\text{depth}(\mathscr{O}_L[\![u]\!])=2,
		\end{equation}
		where \enquote{projdim} and \enquote{depth} respectively denotes the projective dimension and depth of the module $\mathfrak{M}$. Now 
		$\mathfrak{M}$ being a finite free $\mathbb{Z}_p[\![u]\!]$-module, clearly $\{u,p\}$ is a regular sequence for $\mathfrak{M}$ as $\mathbb{Z}_p[\![u]\!]$-module, and so $\{u,p\}$ is a regular sequence for $\mathfrak{M}$ as a $\mathscr{O}_L[\![u]\!]$-module. Therefore the $\text{depth}$ of $\mathfrak{M}$ as $\mathscr{O}_L[\![u]\!]$-module is equal to $2$ i.e., $\text{depth}(\mathfrak{M})=2$. So from equation $ \eqref{e1}$, we get $\text{projdim}(\mathfrak{M})=0$. This means that $\mathfrak{M}$ is projective over the regular local ring $\mathscr{O}_L[\![u]\!]$, and hence $\mathfrak{M}$ is free $\mathscr{O}_L[\![u]\!]$-module. Therefore $\mathfrak{M}$ is $(\mathscr{R} \otimes_{\mathbb{Z}_p} W(\kappa)) [\![u]\!]$-free module. This completes the proof.  
	\end{proof}

	We further continue our investigation toward the projectiveness property of $\mathfrak{M} \in \text{Mod}_{/\mathfrak{S}_\mathscr{R}}^{\varphi}$. For this, we need the help of the following three results:
	\begin{thm} \cite[Theorem~22.3]{Mat} \label{t2.3}
		Let $\mathcal{A}$ be a ring with an ideal $\mathcal{I}$ and an $\mathcal{A}$-module $\mathcal{M}$ such that either $\mathcal{I}$ is a nilpotent ideal or $\mathcal{A}$ is a Noetherian ring and $\mathcal{M}$ is $\mathcal{I}$-adically ideal-separated. Then the following statements are equivalent.
		\begin{enumerate}
			\item[(i)] $\mathcal{M}$ is flat over $\mathcal{A}$;
			\item[(ii)] $\text{Tor}_1^\mathcal{A}(\mathcal{M},\mathcal{N})=0$ for every $\mathcal{A}/\mathcal{I}$-module $\mathcal{N}$;
			\item[(iii)] $\mathcal{M}/\mathcal{I}\mathcal{M}$ is flat over $\mathcal{A}/\mathcal{I}$ and $\mathcal{I} \otimes_\mathcal{A} \mathcal{M}=\mathcal{I} \mathcal{M}$;
			\item[(iv)] $\mathcal{M}/\mathcal{I}\mathcal{M}$ is flat over $\mathcal{A}/\mathcal{I}$ and $\text{Tor}_1^\mathcal{A}(\mathcal{A}/\mathcal{I},\mathcal{M})=0$;
			\item[(v)] $\mathcal{M}/\mathcal{I}\mathcal{M}$ is flat over $\mathcal{A}/\mathcal{I}$ and $\gamma_n:(\mathcal{I}^n/\mathcal{I}^{n+1}) \otimes_{\mathcal{A}/\mathcal{I}}\mathcal{M}/\mathcal{I}\mathcal{M} \to \mathcal{I}^n\mathcal{M}/\mathcal{I}^{n+1}\mathcal{M}$ is an isomorphism for every $n \geq 0$;
			\item[(vi)] $\mathcal{M}/\mathcal{I}\mathcal{M}$ is flat over $\mathcal{A}/\mathcal{I}$ and $\gamma: \oplus_{n \geq 0}\mathcal{I}^n/\mathcal{I}^{n+1} \otimes_{\mathcal{A}/\mathcal{I}} \mathcal{M}/\mathcal{I}\mathcal{M} \to \oplus_{n \geq 0} \mathcal{I}^n/\mathcal{M} \mathcal{I}^{n+1}\mathcal{M}$ is an isomorphism;
			\item[(vii)] $\mathcal{M}/\mathcal{I}^{n+1}\mathcal{M}$ is flat over $\mathcal{A}/\mathcal{I}^{n+1}$ for every $n \geq 0$.
		\end{enumerate}
	\end{thm}
	\begin{lem} \label{l2.4}
		\cite[Tag~0561]{TSP} Let $\mathcal{A} \to \mathcal{B}$ be a ring map of finite type and let $\mathcal{M}$ be a $\mathcal{B}$-module. If $M$ is finitely generated $\mathcal{A}$-module, then it is finitely generated $\mathcal{B}$-module. 
	\end{lem}
	\begin{thm}\cite[Theorem~3.56]{Rot} \label{t2.5}
		Let $\mathcal{A}$ be a commutative ring and $\mathcal{M}$ be a finitely presented $\mathcal{A}$-module, then $\mathcal{M}$ is flat if and only if it is projective.
	\end{thm}
	\begin{thm} \label{t2.6}
		For an arbitrary $\mathbb{Z}_p$-algebra $\mathscr{R}$, if $\mathfrak{M}/u \mathfrak{M}$ is finitely generated projective $\mathscr{R}$-module such that $\mathfrak{M}$ is $u$-adically complete, separated and $u$-torsion free, then $\mathfrak{M}$ is finitely generated and projective as $(\mathscr{R} \otimes_{\mathbb{Z}_p} W(\kappa))[\![u]\!]$-module . 
	\end{thm}
	\begin{proof} 
		Since $\mathfrak{M}$ is a Breuil-Kisin module over $W(\kappa)$, $\mathfrak{M}$ is projective module over $W(\kappa)[\![u]\!]$. For the given $\mathbb{Z}_p$-algebra $\mathscr{R}$, the $(\mathscr{R} \otimes_{\mathbb{Z}_p} W(\kappa))[\![u]\!]$-module $\mathfrak{M}$ will be finitely generated and projective if and only if $\mathfrak{M}/u \mathfrak{M}$ is finitely generated and projective as an $\mathscr{R}$-module. Further, $\mathfrak{M}/u\mathfrak{M}$ is projective module over $(\mathscr{R} \otimes_{\mathbb{Z}_p} W(\kappa))[\![u]\!]$ if and only if it is projective over $\mathscr{R}[\![u]\!]$. On the other hand, the $\mathscr{R}[\![u]\!]$-module $\mathfrak{M}$ is finitely generated and projective once it is $u$-torsion free, $u$-adically complete and separated as well as $\mathfrak{M}/u \mathfrak{M}$ is a finitely generated and projective $\mathscr{R}$-module. But it is equivalent to show that $\mathfrak{M}/u^i \mathfrak{M}, \ i \in \mathbb{N}$ is a projective $\mathscr{R}[u]/u^i$-module. For, if $\mathfrak{M}$ is $u$-adically complete, $u$-adically separated, and $\mathfrak{M}/u^i\mathfrak{M}$ is projective module over $\mathscr{R}[u]/u^i$ for each $i$, and if additionally $\mathfrak{M}/u\mathfrak{M}$ is finitely generated and projective $\mathscr{R}$-module, then $\mathfrak{M}$ is finitely generated and projective over $\mathscr{R}[\![u]\!]$, shown as follows:
		
		Since $\mathfrak{M}$ is $u$-adically complete by assumption and $\mathfrak{M}/u\mathfrak{M}$ is finitely generated, applying the local version of Nakayama lemma, any basis that generates $\mathfrak{M}/u\mathfrak{M}$ lifts to a basis that generates $\mathfrak{M}$, and thus $\mathfrak{M}$ is finitely generated $\mathscr{R}[\![u]\!]$-module. So we get a surjection $\mathscr{R}[\![u]\!]^r \to \mathfrak{M}$ for some $r$. This induces a surjection $\left(\mathscr{R}[u]u^i\right)^r \to \mathfrak{M}/u^i\mathfrak{M}$ for each $i$. Since the target $\mathfrak{M}/u^i\mathfrak{M}$ of these surjections are projective, these surjections split. Such splittings are compatible with $i$ changes. Thus we get a splitting of the original surjection $\mathscr{R}[\![u]\!]^r \to \mathfrak{M}$. Hence $\mathfrak{M}$ is projective over $\mathscr{R}[\![u]\!]$.

		Since $\mathfrak{M}$ is $u$-torsion free, the multiplication map $\mathfrak{M} \to \mathfrak{M}$ by $u$ is injectve. In another word, there exists a short exact sequence of $R$-modules $$ 0 \to \mathfrak{M}/u^i\mathfrak{M} \xrightarrow{u^j} \mathfrak{M}/u^{i+j} \mathfrak{M} \to \mathfrak{M}/u^j \mathfrak{M} \to 0. $$
		Using the conditions (i) and (vii) of Theorem \ref{t2.3}, we conclude that $\mathfrak{M}/u^i\mathfrak{M}$ is a flat module over $\mathscr{R}[u]/u^i$. But $\mathscr{R}[u]/u^i$ is isomorphic to $\mathscr{R}[\![u]\!]/u^i$ and hence $\mathfrak{M}/u^i\mathfrak{M}$ is a flat $\mathscr{R}[\![u]\!]/u^i$-module. We already know that the $\mathscr{R}$-module $\mathfrak{M}/u\mathfrak{M}$ is finitely generated, projective, and hence of finite presentation. For $i=2$, $\mathfrak{M}/u^2\mathfrak{M}$ is an $\mathscr{R}$-module of finite presentation using that fact that $\mathfrak{M}/u\mathfrak{M}$ is an $\mathscr{R}$-module of finite presentation. So by induction on $i$, in the above exact sequence, we get $\mathfrak{M}/u^i\mathfrak{M}$ is an $\mathscr{R}$-module of finite presentation. As $\mathscr{R}[\![u]\!]/u^i$ is a free $\mathscr{R}$-module of rank $i \in \mathbb{N}$, it is finitely generated $\mathscr{R}$-module and hence the ring map $\mathscr{R} \to \mathscr{R}[\![u]\!]/u^i$ is of finite type. Thus $\mathfrak{M}/u^i\mathfrak{M}$ is an $\mathscr{R}[\![u]\!]/u^i$-module of finite presentation by Lemma \ref{l2.4}. Since $\mathfrak{M}/u^i\mathfrak{M}$ is a flat as well as finitely presented $\mathscr{R}[u]/u^i$-module, by Theorem \ref{t2.5} it is projective $\mathscr{R}[\![u]\!]/u^i$-module. Hence $\mathfrak{M}/u^i\mathfrak{M}$ is projective $\mathscr{R}[u]/u^i$-module. This completes the proof. 
	\end{proof}
	\begin{thm} \label{t2.7}
		If $\mathscr{R}$ is a ring such that every finitely generated and $p$-torsion free $\mathscr{R}$-module is projective, then $\mathfrak{M}$ is projective as $(\mathscr{R} \otimes_{\mathbb{Z}_p} W(\kappa)) [\![u]\!]$-module.
	\end{thm}
	\begin{proof}
		Since $\mathfrak{M}/u\mathfrak{M}$ is projective over $W(\kappa)$, it is $p$-torsion free, and therefore projective over $\mathscr{R}$. Hence $\mathfrak{M}$ is projective over $(\mathscr{R} \otimes_{\mathbb{Z}_p} W(\kappa)) [\![u]\!]$ by Theorem \ref{t2.5}.
	\end{proof}
	We note the following result:
	\begin{lem} \label{l2.8}
		\cite[Tag~0AGW]{TSP} Let $\mathcal{I}$ be an ideal of a ring $\mathcal{A}$ and $\mathcal{M}$ be a $\mathcal{A}$-module. Assume that $\mathcal{I}$ is finitely generated, $\mathcal{A}/\mathcal{I}$ is Noetherian ring, $\mathcal{M}/\mathcal{I} \mathcal{M}$ is flat $\mathcal{A}/\mathcal{I}$-module, and $\text{Tor}_{1}^{\mathcal{A}}(\mathcal{M}, \mathcal{A}/\mathcal{I})=0$. Then the $\mathcal{I}$-adic completion $\mathcal{A}^{\wedge}$ is Noetherian and $\mathcal{M}^{\wedge}$ is flat $\mathcal{A}^{\wedge}$-module.
	\end{lem}
	Then we can prove a weaker statement of Theorem \ref{t2.6}:
	\begin{thm} \label{t2.9}
		For an arbitrary Noetherian $\mathbb{Z}_p$-algebra $\mathscr{R}$, if $\mathfrak{M}/u \mathfrak{M}$ is finitely generated projective $\mathscr{R}$-module such that $\mathfrak{M}$ is $u$-adically complete, separated and $u$-torsion free, then $\mathfrak{M}$ is finitely generated and projective as $(\mathscr{R} \otimes_{\mathbb{Z}_p} W(\kappa))[\![u]\!]$-module . 
	\end{thm} 
	\begin{proof}
		Once we note that over the Noetherian ring, every flat module is projective, the proof follows from Lemma \ref{l2.8}.
	\end{proof}
	As a consequence, we have the following weaker statement of Theorem \ref{t2.7}:
	\begin{thm}
		If $\mathscr{R}$ is a Noetherian ring such that every finitely generated and $p$-torsion free $\mathscr{R}$-module is projective, then $\mathfrak{M}$ is projective as $(\mathscr{R} \otimes_{\mathbb{Z}_p} W(\kappa)) [\![u]\!]$-module.
	\end{thm}   
	\begin{proof}
		The proof is immediate from Theorem \ref{t2.7} and Theorem \ref{t2.9}. 
	\end{proof}
	A core in the Breuil-Kisin classification theory of $p$-divisible groups by BK-module is the Frobenious structure on BK-modules, and if a $p$-divisible group has a faithful action of $\mathscr{R}$, then the resulting BK-module attached to $p$-divisible not only has a $\mathscr{R}$-action but also the action is compatible with the Frobenious map $\varphi_{\mathfrak{S}}(u)=u^p$. Exploiting the compatibility of the Frobenious and $\mathscr{R}$-action, we produce the following Theorem \ref{T2.5}. Before proving our result, we recall the definition of fitting ideal of a finitely generated  module and some properties:
	\begin{defn} \cite[Tag~07Z9]{TSP}
		Let $R$ be a ring and $M$ be a finitely generated $R$-module. Choose a presentation $$\oplus_{j \in J}R \to R^{\oplus n} \to M \to 0$$ of $M$, where $J$ is some index set. Let $A=(a_{ij})_{1 \leq i \leq n},~j \in J$ be the matrix of the map  $\oplus_{j \in J}R \to R^{\oplus n}$. The $k^{th}$ fitting ideal $\text{Fitt}_k(M)$ of $M$ is the ideal generated by the $(n-k) \times (n-k)$ minors of $A$, which is independent of the choice of the presentation.
	\end{defn}
	\begin{lem}
		\cite[Proposition~20.8]{DE} \label{L2.4} Let $R$ be a ring and $M$ be a finitely generated $R$-module. Then $M$ is projective of constant rank $r$ if and only if $\text{Fit}_{r-1}(M)=0$ and $\text{Fit}_r(M)=R$. 
	\end{lem}
	\begin{thm} \label{T2.5}
		If $\mathscr{R}$ is a finite flat $\mathbb{Z}_p$-algebra such that $\varphi_{\mathfrak{S}}^{*}\mathfrak{M} \cong \mathfrak{M}$, then $\mathfrak{M}$ is projective.  
	\end{thm}	
	
	\begin{proof}
		Since $\mathfrak{M}$ is finitely generated $\mathfrak{S}_{\mathscr{R}}$-module, we can choose the following presentation for some natural numbers $n,m$:
		$$\mathfrak{S}_{\mathscr{R}}^{\oplus n} \overset{\psi_1}{\longrightarrow} \mathfrak{S}_{\mathscr{R}}^{\oplus m} \overset{\psi_2}{\longrightarrow} \mathfrak{M} \to 0,$$ where the right map $\psi_2$ is surjective by definition. We claim that the above right exact sequence induces an isomorphism when reduced modulo the maximal ideal $\mathfrak{n}$ of $\mathfrak{S}_{\mathscr{R}}$: 
		\begin{align*} & \left(\mathfrak{S}_{\mathscr{R}}/\mathfrak{n}\right)^{\oplus m} \overset{\simeq}{\longrightarrow} \mathfrak{M}/\mathfrak{n}\mathfrak{M} \\
			\text{i.e.,} ~& \mathfrak{S}_{\mathscr{R}}^{\oplus n} \otimes_{\mathfrak{S}_{\mathscr{R}}} \mathfrak{S}_{\mathscr{R}}/ \mathfrak{n} \overset{\simeq}{\longrightarrow} \mathfrak{S}_{\mathscr{R}}^{\oplus m} \otimes_{\mathfrak{S}_{\mathscr{R}}} \mathfrak{S}_{\mathscr{R}}/ \mathfrak{n} \overset{\simeq}{\longrightarrow} \mathfrak{M} \otimes_{\mathfrak{S}_{\mathscr{R}}} \mathfrak{M}/\mathfrak{n} \to 0. \end{align*} So the map $\left(\mathfrak{S}_{\mathscr{R}}/\mathfrak{n}\right)^{\oplus m} \overset{\simeq}{\longrightarrow} \mathfrak{M}/\mathfrak{n}\mathfrak{M}$ must be the zero map, in other words, $\mathfrak{S}_{\mathscr{R}}^{\oplus n}$ goes to $ \mathfrak{n} \mathfrak{S}_{\mathscr{R}}^{\oplus m}$.
		Justification of the claim: As $\mathfrak{M}$ is finite, there is a minimal generating set, say $\{a_1, a_2, \cdots, a_m\}$. Recall the surjective map $$\mathfrak{S}_{\mathscr{R}}^{\oplus m} \overset{\psi_2}{\longrightarrow} \mathfrak{M},~\psi_2(e_i)=a_i,$$ where $\{e_1,e_2, \cdots, e_m\}$ is the basis of $\mathfrak{S}_{\mathscr{R}}^{\oplus m}$. We will show that $\text{ker}(\psi_2) \subset \mathfrak{n}\mathfrak{S}_{\mathscr{R}}^{\oplus m}.$ For let $(b_1,b_2, \cdots, b_m) \in \text{ker}(\psi_2)$, then \begin{align}
			& \nonumber\psi_2 \left((b_1,b_2, \cdots, b_m)\right)=0 \\
			\Rightarrow& b_1a_1+b_2a_2+\cdots+b_ma_m=0. \label{a1}
		\end{align}
		If $b_j \notin \mathfrak{n},$ $1 \leq j \leq m$, then $b_j$ is invertible and hence $a_j= \sum_{i \neq j} b_j^{-1}b_ia_i$ by \eqref{a1}, which implies $a_j$ is redundant in the generating set $\{a_1,a_2, \cdots, a_m\}$ of $\mathfrak{M}$, and this is a contradiction since it was minimal generating set. Therefore $b_j \in \mathfrak{n}$ for all $j$ and so $(b_1,b_2, \cdots, b_m) \in \mathfrak{n}$. Thus $\text{ker}(\psi_2) \subset \mathfrak{n}\mathfrak{S}_{\mathscr{R}}^{\oplus m}.$ Note that $\mathfrak{S}_{\mathscr{R}}$ is Noetherian because $\mathfrak{S}=W(\kappa)[\![u]\!]$ is Noetherian and $\mathscr{R}$ is finite flat $\mathbb{Z}_{p}$-algebra. Let us choose a generating set of $\text{ker}(\psi_2)$ (which exists because $\mathfrak{S}_{\mathscr{R}}$ is Noetherian), then the matrix $A_0=(a_{ij})$ of the map $\psi_1: \mathfrak{S}_{\mathscr{R}}^{\oplus n} \to \mathfrak{S}_{\mathscr{R}}^{\oplus m}$ which sends $\mathfrak{S}_{\mathscr{R}}^{\oplus}$ to $\text{ker}(\psi_2)$ belongs to the maximal ideal $\mathfrak{n}$.
		Since the right exactness property is preserved under base change, applying base change by the Frobenious map $\varphi_{\mathfrak{S}}$ we get the new presentation of $\mathfrak{M}$ as follows:
		\begin{align*}
			&\mathfrak{S}_{\mathscr{R}}^{\oplus n} \overset{\psi_1}{\longrightarrow} \mathfrak{S}_{\mathscr{R}}^{\oplus m} \overset{}{\longrightarrow} \mathfrak{M} \to 0 \\
			\Rightarrow	& \mathfrak{S}_{\mathscr{R}}^{\oplus n} \overset{\varphi_{\mathfrak{S}}(\psi_1)}{\longrightarrow} \mathfrak{S}_{\mathscr{R}}^{\oplus m} \overset{}{\longrightarrow} \mathfrak{M} \otimes_{\varphi_{\mathfrak{S}}, \mathfrak{S}_{\mathscr{R}}} \mathfrak{S}_{\mathscr{R}} \to 0, ~\text{applying base change by $\varphi_{\mathfrak{S}}$} \\
			\Rightarrow & \mathfrak{S}_{\mathscr{R}}^{\oplus n} \overset{\varphi_{\mathfrak{S}}(\psi_1)}{\longrightarrow} \mathfrak{S}_{\mathscr{R}}^{\oplus m} \overset{}{\longrightarrow} \mathfrak{M} \to 0, ~\text{since we assumed $\varphi_{\mathfrak{S}}^{*} \mathfrak{M}=\mathfrak{M} \otimes_{\varphi_{\mathfrak{S}}, \mathfrak{S}_{\mathscr{R}}} \mathfrak{S}_{\mathscr{R}} \cong \mathfrak{M}$}.
		\end{align*} 
		 It is important to note that the maximal ideal $\mathfrak{n}$ is generated by the set $\{u,p\}$. So the matrix entries $a_{ij}$, themselves, are of the form $f(u)=\sum_{i,j\geq 0} p^iu^j$. Here two situations arise when the Frobenious map $\varphi_{\mathfrak{S}}$ map act on the matrix entries $a_{ij}$. 
		 \begin{enumerate}
		 	\item In the 1st case, when the constant term of $a_{ij}$ is zero, the Frobenious map $\varphi_{\mathfrak{S}}(a_{ij})$ lands into $\mathfrak{n}^p$. This happens because if $u \in \mathfrak{n}$ then $u^p \in \mathfrak{n}^p$ i.e., all power series divisible by $u^p$ are in $\mathfrak{n}^p$.  Again repeating the Frobenious map $\varphi_{\mathfrak{S}}(a_{ij}^p)$ lands into $\mathfrak{n}^{p^2}$, and so on.
		 	\item In the 2nd case, when the constant term (which is multiple of $p$) is nonzero, Frobenious map $\varphi_{\mathfrak{S}}$ fixes the constant term but rest other terms maps into $\mathfrak{n}^p$, in other words, Frobenious map $\varphi_{\mathfrak{S}}(a_{ij})$ lands into $\mathfrak{n}+\mathfrak{n}^p \subset \mathfrak{n}$, again repeating the Frobenious map $\varphi_{\mathfrak{S}}(\varphi_{\mathfrak{S}}(a_{ij}))$ lands into $\mathfrak{n}+\mathfrak{n}^{p^2} \subset \mathfrak{n}$, and so on.  
		 \end{enumerate}		
		 Applying base change by the Frobenious map $\varphi_{\mathfrak{S}}$ repeatedly, we get after $N^{th}$ step, $\varphi_{\mathfrak{S}}(a_{ij}) \in \bigcap_{i=1}^{N} \mathfrak{n}^{N}$. Now, let the $k^{th}$ fitting ideal of $\mathfrak{M}$ be $\text{Fitt}_k(\mathfrak{M})$ for $k<m$. Since the fitting ideals of $\mathfrak{M}$ are generated by the minors of the map $\mathfrak{S}_{\mathscr{R}}^{\oplus n} \to \mathfrak{S}_{\mathscr{R}}^{\oplus m}$, we conclude $$\text{Fitt}_k(\mathfrak{M}) \subset \bigcap_{N \geq 1} \mathfrak{n}^N~\text{for all}~k<m.$$ 
		By Krull's intersection theorem, $\bigcap_{N \geq 1} \mathfrak{n}^N=0$, which implies all the fitting ideals of $\mathfrak{M}$ are trivial for $k<m$ i.e., $\text{Fit}_k(\mathfrak{M})=0$ for $k<m$, and since $\mathfrak{M}$ has $m$ generators $\{a_1,a_2, \cdots, a_m\}$ by definition $\text{Fit}_m(\mathfrak{M})=\mathfrak{S}_{\mathscr{R}}.$ Hence by \ref{L2.4}, $\mathfrak{M}$ is projective $\mathfrak{S}_{\mathscr{R}}$-module.
	\end{proof}

	\subsection*{Acknowledgement} The authors are grateful to Olivier Brinon for valuable comments in an old version of this paper. The first author acknowledges \textit{CSIR}, Government of India, for the award of Senior Research Fellowship with File no.-09/025(0249)/2018-EMR-I.

\end{document}